\documentclass[psfceqn,leqno]{amsart}
\usepackage{amsfonts}
\usepackage{amsthm}
\usepackage{amsmath}
\usepackage{amssymb}
\usepackage{color}
\usepackage{mathrsfs}
\usepackage{tikz}
\usetikzlibrary{arrows}
\usetikzlibrary{graphs}

\newtheorem{theorem}{Theorem}[section]
\newtheorem{corollary}[theorem]{Corollary}
\newtheorem{proposition}[theorem]{Proposition}
\newtheorem{lemma}[theorem]{Lemma}

\theoremstyle{definition}
\newtheorem{definition}[theorem]{Definition}
\newtheorem{remark}[theorem]{Remark}
\newtheorem{example}[theorem]{Example}

\def\N{{\mathbb N}}

\def\Z{{\mathbb Z}}

\numberwithin{equation}{section}
    \thanks{\hspace{-1.66em} 2010 \emph{Mathematics Subject Classification.}
	Primary 37B99; Secondary 54H20.}
   \thanks{\noindent $^\dagger$The paper was partially written
	when the first author visited the National University of Ireland,
	Galway, in July 2013 when he was on sabbatical leave. He would 
	like to acknowledge the support by NUI Galway Millennium Fund 
	and the hospitality of the School of Mathematics, Statistics 
	and Applied Mathematics at NUI, Galway.}

\begin{document}
    \title[Topological Transitivity in Quasi-continuous $\cdots$]{Topological Transitivity in Quasi-continuous 
    Dynamical Systems$^\dagger$}
    \author[J. Cao]{Jiling Cao}
    \address{Department of Mathematical Sciences, School of Engineering, Computer and Mathematical Sciences, 
    Auckland University of Technology, Private Bag 92006, Auckland 1142, New Zealand}
    \email{jiling.cao@aut.ac.nz}
    \author[A. McCluskey]{Aisling McCluskey}
    \address{School of Mathematics, Statistics and Applied
    Mathematics, National University of Ireland, Galway, Ireland}
    \email{aisling.mccluskey@nuigalway.ie}
    \date{}
    \maketitle

    \begin{abstract}
    A quasi-continuous dynamical system is a pair $(X,f)$ consisting of a topological space $X$ and a mapping $f: 
    X\to X$ such that $f^n$ is quasi-continuous for all $n \in \N$, where $\mathbb N$ is the set of 
    non-negative integers. In this paper, we show that under appropriate assumptions, various definitions 
    of the concept of topological transitivity are equivalent in a quasi-continuous dynamical system. Our 
    main results establish the equivalence of topological and point transitivity in a quasi-continuous 
    dynamical system. These extend some classical results on continuous dynamical systems in 
    \cite{akin-carlson:2012}, \cite{degirmenci:2003} and \cite{silverman:1992}, and some results on 
    quasi-continuous dynamical systems in \cite{crannell-frantz:05} and  \cite{crannell-martelli:00}. 
    \end{abstract}

    %%%%%%%%%%%%%%%%%%%%%%%%%%%%%%%%%%%%%%%%%%%%%
    \section{Introduction} \label{sec:intro} 
    %%%%%%%%%%%%%%%%%%%%%%%%%%%%%%%%%%%%%%%%%%%%%
    
    In the literature, two groups of different definitions of chaotic dynamical systems have been proposed. 
    In the first group, chaos is approached from the measure theoretic point of view. In the second group, 
    chaos is approached from the non-linear analysis point of view, where a mapping $f: X \to X$ is 
    considered \emph{chaotic} in $X$ if $f$ has at least sensitive dependence on initial conditions in 
    $X$. To this requirement, many authors add topological transitivity, and another condition frequently 
    required is the existence of a dense orbit, see \cite{devaney:1989}, \cite{moser:1980}, 
    \cite{gulick:1992}, \cite{martelli:1992}, \cite{martelli-dang:1998} and \cite{wiggins:1991}. The 
    latter property has also been called \emph{point transitivity} by some authors.
    
    %\medskip
    As a motivation for the notion of topological transitivity of a system, one may think of a real 
    physical system, where a state is never given or measured exactly, but always up to a certain 
    error. So, instead of points, one should study (small) open subsets of the phase space and describe 
    how they move in that space. Intuitively, a topologically transitive mapping $f$ has points that
    eventually move under iteration from one arbitrarily small neighbourhood to any other. Consequently, 
    the dynamical system cannot be broken down or decomposed into two subsystems (disjoint sets with 
    nonempty interiors) which do not interact under $f$, i.e., are invariant under the mapping.
    
    %\medskip
    In the study of dynamical systems, it is generally assumed that ``topological transitivity" and 
    ``point transitivity" are equivalent when $X$ is a compact metric space and $f$ is continuous, e.g., 
    Proposition 39 of \cite{block-coppel:1992}. However, as noted in \cite{kolyada-snoha:94} and 
    \cite{degirmenci:2003}, these two conditions are independent in general even in compact metric spaces 
    with continuous mappings. Moreover, there are several different common definitions of topological 
    transitivity. It had been a part of the folklore of dynamical systems that under reasonable 
    assumptions they are equivalent until Akin and Carlson \cite{akin-carlson:2012} provided a complete 
    description of the relationships among them. In \cite{akin-auslander:2016}, Akin et al. further 
    described various strengthenings of the concept of topological transitivity. 
    
    %\medskip
    A common framework in the study of dynamical systems assumes that the phase space is compact
    metric and the self-mapping is continuous. One of the difficulties in relating the definitions 
    of the two groups, mentioned at the beginning, derives from a topological property of $f$. This 
    is because non-linear definitions normally require the continuity of $f$, but measure theoretic 
    definitions may apply to functions with some type of discontinuity which is not too far from 
    continuity. Motivated by this, Crannell and Martelli studied dynamics of quasi-continuous 
    mappings in \cite{crannell-martelli:00}. They showed the equivalence for quasi-continuous 
    mappings of two non-linear analysis definitions of chaotic dynamical systems due to Wiggins 
    \cite{wiggins:1991} and Martelli \cite{martelli:1992}. They also extended several well known
    results in \cite{banks:1992} and \cite{touhey:1997} for continuous dynamical systems to 
    quasicontinuous systems. Note that Crannell and Martelli \cite{crannell-martelli:00} assumed
    the phase spaces of their systems to be compact metric. 
    
    %\medskip
    In this paper, we continue the study of dynamics of quasi-continuous systems. Our motivation 
    is to study relationships among various definitions of topological transitivity in quasi-continuous
    dynamical systems whose phase spaces are general topological spaces. 
    The rest of this paper is organized as follows. In Section \ref{sec:notion}, we introduce notation, 
    definitions and basic relationships among different concepts of topological transitivity and point
    transitivity. In Section \ref{sec:quasi-cont}, we introduce quasi-continuous dynamical 
    systems and study some basic properties. We also provide two results which show how far a 
    quasi-continuous system is from a continuous one. Section \ref{sec:equivalence} is devoted to a
     study of equivalence among different versions of topological transitivity, and  equivalence of
    topological and point transitivity in a quasi-continuous dynamical system. In the 
    last section, we discuss what happens with these equivalences in a quasi-continuous dynamical 
    system when the phase space contains isolated points.

    %%%%%%%%%%%%%%%%%%%%%%%%%%%%%%%%%%%%%%%%%%%%%%%%%%%%%%%%%%%%%%%%%%%%%%%%%%%%%%%%%%%%%
    \section{Definitions and basic relationships} \label{sec:notion} 
    %%%%%%%%%%%%%%%%%%%%%%%%%%%%%%%%%%%%%%%%%%%%%%%%%%%%%%%%%%%%%%%%%%%%%%%%%%%%%%%%%%%%%
    
    Let $\mathbb N$ denote the set of nonnegative integers and let $\mathbb Z$ denote the set of
    integers. By a dynamical system, we mean a pair $(X, f)$, where $X$ is a topological space 
    (called the \emph{phase space}) and $f: X \to X$ is a mapping from $X$ into itself (not necessarily 
    continuous). The dynamics of the system is given by iteration. To avoid triviality, throughout
    the paper, we assume that $X$ contains at least two points. A point $x \in X$ ``moves", with 
    its trajectory being the sequence $x$, $f(x)$, $f^2(x)$, $f^3(x)$, $\dots$, where $f^n$ is 
    the $n$th iteration of $f$. The point $f^n(x)$ is the position of $x$ after $n$ units of time. 
    The set of points of the trajectory of $x$ under $f$ is called the \emph{forward orbit of $x$}, 
    denoted by ${\rm Orb}_f(x)$, that is, ${\rm Orb}_f(x) =\{f^n(x): n \in \mathbb N\}$.
    The \emph{omega limit set for $x$ under $f$}, denoted by $\omega f(x)$, is given by
    \[
    \omega f(x) := \bigcap_{k\in \mathbb N}\overline{\{f^n(x): n \geq k\}},
    \] 
    and is precisely the set of all accumulation points of the sequence $\langle f^n(x): n \in \mathbb N \rangle $. A bi-infinite
    sequence $\langle x_k: k \in \mathbb Z \rangle$ is called an \emph{orbit sequence} if 
    $f(x_k)=x_{k+1}$ for all $k$; and the set  $\{ x_k: k \in \mathbb Z \}$ of its elements is
    called an \emph{orbit}. In addition, we will also call a sequence $\langle x_k: k \ge n 
    \rangle$ an \emph{orbit sequence} if $n \in \N$ and $f(x_k) =x_{k+1}$ for all $k \ge n$ 
    and $f^{-1}(x_n)=\emptyset$; the set of elements of this sequence is ${\rm Orb}_f(x_n)$.
    
    %\medskip
    For subsets $A, B \subseteq X$, following \cite{akin-carlson:2012}, we define 
    the \emph{hitting time sets}
    \[
    N(A, B) := \{n \in \mathbb Z: f^n(A) \cap B \ne \emptyset\}\ \mbox{ and }\
    N_+(A, B) := N(A, B) \cap {\mathbb N}.
    \]
    Note that
    \[
    N(A, B) = N_+(A, B) \cup N_+(B,A).
    \]
    
    \begin{definition} \label{defn:toptrans}
    A dynamical system $(X, f)$ is called \emph{topologically transitive} (TT$_+$) if for every 
    pair of nonempty open sets $U, V \subseteq X$, the set $N_+(U, V)$ is nonempty. 
    \end{definition}
    
    It is easy to see that $(X, f)$	is topologically transitive if, and only if, for every nonempty 
    open set $U$, $\bigcup_{n\in {\mathbb N}}f^n(U)$ is dense in $X$. 
    
    \begin{remark}
    Definition \ref{defn:toptrans} is the definition for topological transitivity commonly given 
    in the literature e.g., \cite{akin-auslander:2016}, \cite{degirmenci:2003}, \cite{kolyada-snoha:94}
    and \cite{silverman:1992}. 
    Note that Akin and Carlson \cite{akin-carlson:2012} define the properties TT and IN in the system 
    $(X, f)$ as follows:
    
    \medskip
    \begin{itemize}
    \item[(IN)] $X$ is not the union of two proper, closed and +invariant sets, where a set $A\subseteq
    X$ is called \emph{+invariant} if $f(A) \subseteq A$.
    \item[(TT)] For every pair of nonempty open sets $U, V \subseteq X$, the set $N(U, V)$ is nonempty.
    \end{itemize}

    \medskip\noindent
    They labelled topological transitivity in Definition \ref{defn:toptrans} as the property TT$_+$. 
    In the same paper, they also defined the property TT$_{++}$ as follows:
    
    \medskip
    \begin{itemize}
    \item[(TT$_{++}$)] For every pair of nonempty open sets $U, V \subseteq X$, the set $N_+(U, V)$ is 
    infinite.
    \end{itemize}
    \end{remark}
    
    A point $x \in X$ is called a \emph{transitive point} when for every nonempty open $V \subseteq X$, 
    the hitting time set $N_+(\{x\}, V)$ is nonempty. This is equivalent to saying that ${\rm Orb}_f(x)$
    is dense. The set of transitive points of $(X,f)$ is denoted by Trans$_f$. Following 
    \cite{akin-carlson:2012}, we define the following properties:
    
    \medskip
    \begin{itemize}
    \item[(DO)] There is an orbit sequence $\langle x_k: k \in \mathbb Z \rangle$ or $\langle x_k: k
    \ge n \rangle$ (for some $n$) dense in $X$.
    
    \item[(DO$_{+}$)] There is a point $x\in X$ such that ${\rm Orb}_f(x)$ is dense.
    
    \item[(DO$_{++}$)] There is a point $x\in X$ such that $\omega f(x) =X$.
    \end{itemize}
    
    \begin{definition} \cite{akin-carlson:2012}
    A dynamical system $(X, f)$ is called \emph{point transitive} if DO$_{+}$ holds in $(X,f)$.
    \end{definition}
    
    %\medskip
    Akin and Carlson \cite{akin-carlson:2012} established the following implications for any general
    dynamical system $(X,f)$:
    \[
    \begin{tikzpicture}
    %\begin{scope}[>=latex]
    %\draw[->] (-2,0) -- (2,0); % x-axis
    %\draw[->] (0,-2) -- (0,2); % y-axis
    %\end{scope}
    
    \node (a) at (-2,2) {DO$_{++}$};
    \node (b) at (0,2)  {DO$_+$};
    \node (c) at (2,2)  {DO};
    \graph {(a) -> (b) -> (c)};
    
    \node (d) at (-2, 0.2) {TT$_{++}$};
    \node (e) at (0, 0.2) {TT$_{+}$};
    \node (f) at (2, 0.2) {TT};
    \node (g) at (4, 0.2) {IN};
    \graph {(d) -> (e) -> (f) <-> (g)};
    
    \draw[->] (-2,1.5) -- (-2,0.7);
    \draw[->] (2,1.5) -- (2,0.7);
    
    \node (h) at (0, -0.5) {Diagram 1.};
    
    \end{tikzpicture}
    \]
    It was shown in Theorem 1.4 of \cite{akin-carlson:2012} if $X$ is a perfect (i.e. without isolated points) and $T_1$ space, then DO$_{++}$ and DO$_{+}$ are equivalent. Furthermore, if 
    $X$ is a perfect and $T_1$ space and $f$ is a continuous mapping, then TT$_{++}$, 
    TT$_{+}$ and TT are equivalent, refer to Proposition 4.2 in \cite{akin-carlson:2012}.

    %%%%%%%%%%%%%%%%%%%%%%%%%%%%%%%%%%%%%%%%%%%%%%%%%%%%%%%%%%%%%%%%%%%%%%%%%%%%%%%
    \section{Quasi-continuous dynamical systems} \label{sec:quasi-cont} 
    %%%%%%%%%%%%%%%%%%%%%%%%%%%%%%%%%%%%%%%%%%%%%%%%%%%%%%%%%%%%%%%%%%%%%%%%%%%%%%%
    
    Suppose that $X$ and $Y$ are topological spaces and $f: X \to Y$ is a mapping.  We say that 
    $f$ is \emph{quasi-continuous at a point $x_0 \in X$} if, for each open neighbourhood $W$ of 
    $f(x_0)$ and each open neighbourhood $U$ of $x_0$, there exists a nonempty open subset $V$ of $U$ 
    such that $f(V) \subseteq W$. If $f$ is quasi-continuous at each point of $X$, then we say that 
    $f$ is \emph{quasi-continuous on $X$}.
    
    %\medskip
    This notion informally appeared in Baire's PhD thesis \cite{baire:1899}, where he indicated 
    that it was suggested to him by Volterra. Later, the notion of a quasi-continuous mapping 
    was formally introduced/defined by Kempisty \cite{kempisty:32} for real-valued functions of 
    real variables. Quasi-continuity of mappings between general topological spaces was also studied 
    by Levine \cite{levine:63} under the name of semi-continuity.
    
    %\medskip
    As noted by Crannell and Martelli in \cite{crannell-martelli:00}, the composite of two 
    quasi-continuous mappings may fail to be quasi-continuous. Thus, when we try to extend results
    in a dynamical system $(X,f)$ where $f$ is continuous, we cannot just simply relax $f$ to be
    quasi-continuous. This motivates the following definition.
    
    \begin{definition}[\cite{crannell-martelli:00}]
    A system $(X,f)$ is said to be \emph{quasi-continuous} if for every $n\in \mathbb N$, $f^n: X 
    \to X$ is quasi-continuous.
    \end{definition}

    There is a quasi-continuous dynamical system which is not continuous, as shown by the following 
    simple example.

    \begin{example} \label{exam:quas-cont}
    Let $X=[0, 1]$ be endowed with the usual topology. Define a mapping $f: X \to X$ by
    \[
    f(x) = \left\{
    \begin{array}{ll}
    0, & \mbox{if $0 \le x \le \frac{1}{2}$;}\\[0.5em]
    1, & \mbox{if $\frac{1}{2} < x \le 1$.}
    \end{array}
    \right.
    \]
    This mapping $f$ is continuous at any point $x \not= \frac{1}{2}$ and is quasi-continuous at $x=
    \frac{1}{2}$. So $f$ is quasi-continuous but not continuous on $X$. Furthermore it can be
    readily checked that $f^n=f$ for all $n\in \N$. Thus we conclude that $(X,f)$ is a 
    quasi-continuous but not continuous dynamical system. 
    \end{example}

    The following lemma gives a characterization of quasi-continuity which is a useful tool in the 
    study of quasi-continuous dynamical systems. The proof of this lemma is straightforward and
    can be found in \cite{crannell-martelli:00} or other references.
    
    \begin{lemma}[\cite{crannell-martelli:00}] \label{lemma:key}
    Let $X$ and $Y$ be topological spaces and  $f:X \to Y$ a mapping. Then $f$
    is quasi-continuous on $X$ if, and only if, for each pair of nonempty open sets $U \subseteq X$
    and $V \subseteq Y$, either $U$ and $f^{-1}(V)$ are disjoint or there is a nonempty open subset 
    $W \subseteq X$ such that $W \subset U \cap f^{-1}(V)$. 
    \end{lemma}
    
    Our next result is analogous to Lemma 4.1 in \cite{akin-carlson:2012}.
    
    \begin{proposition} \label{prop:infinite}
    Let $(X,f)$ be a quasi-continuous dynamical system where $X$ is a perfect and Hausdorff phase space.
    If $(X,f)$ satisfies ${\rm TT}$, then for any nonempty open subset $U \subseteq X$, $N_+(U, U)$ 
    is infinite.
    \end{proposition}

    \begin{proof}
    Let $U\subseteq X$ be a nonempty open subset. We will define inductively a nested sequence 
    $\langle U_n: n \in\mathbb N\rangle$ of nonempty open subsets of $U$ and a strictly increasing 
    sequence $\langle k_n: n \in \mathbb N\rangle$ in $\mathbb N$ such that $f^{k_n} (U_n)\subseteq U$.
    
    %\medskip
    %\noindent
    \emph{Initial step}. Let $U_0= U$ and $k_0=0$. It is trivial that $f^{k_0} (U_0) \subseteq U$.
    
    %\medskip
    %\noindent
    \emph{Induction step.} Suppose that we have defined a finite sequence 
    \[
    U_0 \supseteq U_1 \supseteq \cdots \supseteq U_{n-1} \supseteq U_n
    \] 
    of nonempty open subsets in $U$ and a finite sequence of integers $k_0<k_1 < \cdots <k_n$ in $\N$
    such that $f^{k_n}(U_n) \subseteq U$. Since $X$ is perfect and Hausdorff, we can pick up two 
    distinct points $x_{n+1}$ and $y_{n+1}$ in $U_n$ and two disjoint nonempty open subsets $V_{n+1}$ 
    and $W_{n+1}$ in $U_n$ such that $x_{n+1} \in V_{n+1}$ and $y_{n+1}\in W_{n+1}$. Since $(X,f)$ 
    satisfies TT, there is an integer $i_{n+1}$ such that 
    \[
    i_{n+1} \in N_+(V_{n+1}, W_{n+1})\cup N_+(W_{n+1}, V_{n+1}).
    \] 
    Without loss of generality, we assume that $i_{n+1} \in N_+(V_{n+1}, W_{n+1})$. By Lemma 
    \ref{lemma:key}, there is a nonempty open subset $U_{n+1}$ such that 
    \[
    U_{n+1} \subseteq V_{n+1} \cap f^{-i_{n+1}}(W_{n+1}). 
    \]
    Now it is clear that $U_{n+1} \subseteq U_n$. Also, as $V_{n+1}$ and $W_{n+1}$ are disjoint, 
    we must have $i_{n+1}>0$. Put $k_{n+1} = k_n + i_{n+1}$. Then $k_n < k_{n+1}$ and 
    \[
    f^{k_{n+1}}(U_{n+1}) = f^{k_n}\left(f^{i_{n+1}} (U_{n+1})\right) \subseteq f^{k_n}(W_{n+1}) \subseteq
    f^{k_n}(U_n) \subseteq U.
    \]
    This completes the induction step.
    
    %\medskip
    Finally, by the construction of sequences $\langle U_n: n \in\mathbb N\rangle$ and $\langle k_n: 
    n \in \mathbb N\rangle$, we have $k_n \in N_+(U,U)$ for each $n \in \mathbb N$. We conclude 
    that $N_+(U,U)$ is infinite.
    \end{proof}

    A natural question is: \emph{How far is a quasi-continuous system $(X,f)$ from a continuous system?} 
    To study this question, we define $C^{\infty}(f)$ and $C^{\infty}_f$ as follows:
    \[
    C^{\infty}(f) := \{x \in X: f^n \mbox{ is continuous at $x$ for all } n \in \N \}
    \]
    and
    \[
    C^{\infty}_f := \{x \in X: f^n(x) \in C(f)  \mbox{ for all } n \in \N
    \},
    \]
    where $C(f)$ is the set of points at which $f$ is continuous.
    If $x \in C^{\infty}_f$, then $f$ is continuous at every point along ${\rm Orb}_f(x)$,
    and accordingly, $f^n$ is continuous at $x$ for every $n \in \mathbb N$, i.e. $C^{\infty}_f
    \subseteq C^{\infty}(f)$.
    
    %\medskip
    Let $X$ be a topological space, and let $\rho$ be a metric on $X$. Then $X$ is said to be 
    \emph{fragmentable by $\rho$} if for all $\epsilon > 0$ and every non-empty $A \subseteq X$, 
    there is a relatively open non-empty $B \subseteq A$ with $\rho$-$\mbox{diam}(B) < \epsilon$.
    Note that every metrizable space $X$ is fragmented by some metric $\rho$. We refer 
    the reader to \cite{jayne:1993} for more details on fragmentability.
    Recall that a subset $R$ of $X$ is called \emph{residual}, if $X\setminus R$ is a countable
    union of nowhere dense subsets of $X$, or equivalently, $R$ contains a countable intersection of
    dense open subsets. Moreover if $X$ is a Baire space, then any residual subset $R$ in $X$ contains
    a dense $G_\delta$-set of $X$.
    
    \begin{proposition} \label{prop:residual1}
    Let $(X, f)$ be a quasi-continuous system. If $X$ is fragmented by a metric $\rho$ such that the 
    topology generated by the metric $\rho$ contains the topology of the space $X$, then 
    $C^{\infty}(f)$ is a residual set in $X$. Furthermore, if $X$ is also a Baire space, then 
    $C^{\infty}(f)$ contains a dense $G_\delta$-set.
    \end{proposition}
    
    \begin{proof} 
    For each $n\in \mathbb N$, since $f^n$ is quasi-continuous, by Theorem 1 in \cite{KKM:01b}, 
    $C(f^n)$ is a residual subset of $X$. Note that 
    \[
    C^{\infty}(f) = \bigcap_{n \in \mathbb N}C(f^n).
    \]
    Thus, $C^{\infty}(f)$ is also a residual subset of $X$.
    \end{proof}
    
    To see when $C^{\infty}_f$ is residual, we need some notion of openness on $f$. Given topological spaces $Y$ and $Z$,  recall that a 
    mapping $f: Y \to Z$ is \emph{feebly open} 
    \cite{haworth:1977} if for every nonempty open set $U \subseteq Y$, the interior of $f(U)$ is 
    nonempty.  Crannell et al \cite{crannell-frantz:05} called a quasi-continuous and feebly 
    open mapping \emph{quopen}. Theorem 8 of \cite{crannell-frantz:05} asserts that if $X$ is a compact
    metric space and $f$ is quopen, then $C^{\infty}_f$ is a residual set in $X$. To extend this result,
    we employ the concept of a $\delta$-open mapping introduced by Haworth and McCoy \cite{haworth:1977}.
    
    \begin{definition}[\cite{haworth:1977}]
    Given topological spaces $Y$ and $Z$, a mapping $f: Y \to Z$  is called
    \emph{$\delta$-open} if for every nowhere dense subset $N$ of $Z$, $f^{-1}(N)$ is nowhere dense 
    in $Y$, or equivalently, for every somewhere dense subset $A$ of $Y$, $f(A)$ is a somewhere dense 
    subset of $Z$.
    \end{definition}

    %\medskip
    The following proposition may be known, but we cannot find it in the literature. 
    For the sake of completeness, we provide a full proof here.
    
    \begin{proposition} \label{lemma:weak-open}
    If a mapping $f: Y \to Z$ is quasi-continuous and feebly open, then it is $\delta$-open.
    \end{proposition}
    
    \begin{proof}
    Let $N \subseteq Z$ be a nowhere dense subset of $Z$. To derive a contradiction, we assume that
    $f^{-1}(N)$ is somewhere dense. Then there exists a nonempty open subset $U$ of $Y$ such that $U 
    \subseteq \overline{f^{-1}(N)}$. Since $f$ is a feebly open mapping, then 
    $V:={\rm int}(f(U)) \ne \emptyset$.
    Furthermore, since $N$ is nowhere dense, the set
    \[
    W:=V \cap (Z \setminus \overline{N})
    \]
    is nonempty open in $Z$. Now we have a nonempty open subset $U$ of $Y$ and a nonempty open subset
    $W$ of $Z$ with $U \cap f^{-1}(W) \ne \emptyset$. Since $f$ is quasi-continuous, by Lemma 
    \ref{lemma:key} we have a nonempty open subset $U'$ such that $U'
    \subseteq U \cap f^{-1}(W)$. On the one hand, $U' \subseteq U \subseteq
    \overline{f^{-1}(N)}$ implies that $f(U') \cap N \ne \emptyset$. On the other hand, $f(U') 
    \subseteq W$ implies that $f(U') \cap N = \emptyset$. We have reached a contradiction. Thus 
    $f^{-1}(N)$ must be nowhere dense.
    \end{proof}

    Our next result extends Theorem 8 in \cite{crannell-frantz:05}.
    
    \begin{proposition} \label{prop:residual2}
    Consider a dynamical system $(X, f)$. If
    \begin{enumerate}
    \item $f$ is quasi-continuous and $\delta$-open, and
    \item $X$ is fragmented by a metric $\rho$ such that the topology generated by the metric 
    $\rho$ contains the topology of the space $X$,
    \end{enumerate} 
    then $C^{\infty}_f$ is a residual set in $X$. Furthermore if $X$ is also a Baire space, then 
    $C^{\infty}_f$ contains a dense $G_\delta$-set.
    \end{proposition}

    \begin{proof}
    First, as in \cite{crannell-frantz:05}, we can easily show that
    \[
    C^{\infty}_f = \bigcap_{n \in \mathbb N} f^{-n}(C(f)).
    \]
    By Theorem 1 in \cite{KKM:01b},  $C(f)$ is residual in $X$. Since
    $f$ is $\delta$-open, then $f^{-1}(C(f))$ is also residual. By induction, we see that $ f^{-n}(C(f))$
    is residual for all $n \in \mathbb N$. It follows that $C^{\infty}_f$ is a residual set in $X$.
    \end{proof}

    Propositions \ref{prop:residual1} and \ref{prop:residual2} indicate that in a certain sense, a
    quasi-continuous dynamical system approximates some continuous dynamical system. In \cite{akin:2013}
    and \cite{crannell-frantz:05}, the dynamics of a quasi-continuous mapping was described in terms 
    of suitable closed relations, and connected with the continuous dynamics on an
    invariant $G_\delta$-set and with continuous dynamics on the compact space of sample paths.
    
    \begin{remark}
    Note that the conclusion of Proposition \ref{prop:residual1} does not hold if we only require
    that $f$ is quasi-continuous, even when $X$ is a compact metric space. Indeed Crannell and Sohaib
    \cite{crannell-sohaib:08} provide an example of a quasi-continuous function $f: [0, 2] \to [0, 2]$  
    such that $f^2$ is not continuous at any point of $[0,2]$.
    \end{remark}

    %We conclude this section with the following open question.
    
    %\begin{question}
	%Does there exist a quasi-continuous dynamical system $(X,f)$ such that $f$ is not continuous? 
    %\end{question}

    %%%%%%%%%%%%%%%%%%%%%%%%%%%%%%%%%%%%%%%%%%%%%%%%%%%%%%%%%%%%%%%%%%%%%%%%%%%%%%%%%%%%%%%%%%%%%%%%%
    \section{Equivalence theorems in quasi-continuous dynamical\\ 
    systems with perfect phase space}\label{sec:equivalence} 
    %%%%%%%%%%%%%%%%%%%%%%%%%%%%%%%%%%%%%%%%%%%%%%%%%%%%%%%%%%%%%%%%%%%%%%%%%%%%%%%%%%%%%%%%%%%%%%%%%
    
    In this section, we establish the equivalence between point and topological transitivity in a
    quasi-continuous dynamical system whose phase space is perfect. 
    
    The following result is an analogue of Proposition 4.2 of \cite{akin-carlson:2012}.
    
    \begin{theorem} \label{thm:transitivity}
    Let $(X,f)$ be a quasi-continuous dynamical system. If $X$ is a perfect and Hausdorff space, 
    then the following implications hold:
    \[
    {\rm TT} \longrightarrow \mbox{\rm TT$_{+}$} \longrightarrow \mbox{\rm TT$_{++}$}.
    \]
    \end{theorem}
    
    \begin{proof}
    It suffices to show that TT implies TT$_{++}$. Assume that $(X,f)$ satisfies TT. To show that
    $(X,f)$ satisfies TT$_{++}$, let $V$ and $W$ be any pair of nonempty open subsets of $X$. By 
    the TT property, there exists $n \in \mathbb Z$ such that $f^n(V) \cap W \ne \emptyset$. 
    
    \emph{Case 1. $n \ge 0$.} Since $f^n$ is quasi-continuous, there is a nonempty open subset
    $U$ of $X$ satisfying $U \subseteq V \cap f^{-n}(W)$. Then, by Proposition \ref{prop:infinite},
    $N_+(U,U)$ is infinite. For any $k\in N_+(U,U)$, we have
    \[
    \emptyset \ne f^n\left(f^k(U) \cap U\right) \subseteq f^{n+k}(U) \cap f^{n}(U) \subseteq f^{n+k}(V)
    \cap W
    \]
    and hence $n+k \in N_+(V, W)$. Thus, $N_+(V,W)$ is infinite.
    
    %\medskip
    \emph{Case 2. $n < 0$.}
    Since $f^{-n}$ is quasi-continuous, there is a nonempty open subset $U$ of $X$ satisfying 
    $U \subseteq f^n(V) \cap W$. Then, by Proposition \ref{prop:infinite}, $N_+(U,U)$ is infinite. 
    For any $k\in N_+(U,U)$ with $k >-n$, we have
    \[
    \emptyset \ne f^k(U) \cap U \subseteq f^k\left(f^n(V)\right) \cap W \subseteq f^{n+k}(V)
    \cap W,
    \]
    and hence $n+k \in N_+(V, W)$. Thus, $N_+(V,W)$ is infinite.
    \end{proof}
    
    %\medskip
    Let $\mathscr P$ be a family of nonempty open subsets in a topological space $X$. We call 
    $\mathscr P$ a \emph{$\pi$-base for $X$} if for every nonempty open subset $U$ of $X$, there 
    exists some $P \in \mathscr P$ such that $P \subseteq U$.
    
    \begin{theorem} \label{thm:dense_orbit}
    Let $(X,f)$ be a quasi-continuous dynamical system. Suppose that $X$ is a Baire space with 
    a countable $\pi$-base $\{P_n: n \geq 0\}$. Then the following conditions are equivalent:
    
    \medskip
  	\begin{itemize}
  	\item[(1)] {\rm DO}$_{++}$.
  	\item[(2)] For any non-empty open subset $U \subseteq X$ and any $k\in \mathbb N$, 
  	$\bigcup_{n \geq k} {\rm int}(f^{-n}(U))$ is dense in $X$.
  	\item[(3)] The set $\{ x\in X: \omega f(x)=X\}$ contains a dense $G_\delta$-set of $X$.
  	\end{itemize}
    \end{theorem}
  
    \begin{proof}
  	Since (3) $\rightarrow$ (1) is trivial, we need only to prove (1) $\rightarrow$ (2) 
  	and (2) $\rightarrow$ (3).
  	
  	%\medskip
  	(1) $\rightarrow$ (2). Let a nonempty open set $U\subseteq X$ and a $k\in \mathbb N$ be 
  	given. Let $V$ be an arbitrary nonempty open subset of $X$. By the {\rm DO}$_{++}$ property, 
    there exists some $n \geq k$ such that $f^n(x) \in V$. 
  	In addition, the {\rm DO}$_{++}$ property also implies that
  	\[
  	U\cap \left\{ f^m(x): m \ge n+k+1\right\} \ne \emptyset.
  	\]
    Thus there exists an $m > n + k$ such that $f^m(x) \in U$. Since 
  	$f^{(m-n)}$ is quasi-continuous and $f^n(x) \in V \cap f^{-(m-n)}(U)$, by Lemma \ref{lemma:key}
  	there exists a nonempty open subset $W$ of $X$ such that $W \subseteq V \cap f^{-(m-n)}(U)$. 
  	Thus,
  	\[
  	\emptyset \not = W \subseteq V \cap {\rm int}\left(f^{-(m-n)}(U)\right) 
  	\subseteq V \cap \left(\bigcup_{n\geq k} {\rm int}\left(f^{-n}(U)\right) \right),
  	\]
  	which implies that (2) holds.
  	
  	%\medskip
  	(2) $\rightarrow$ (3). Recall that $\{P_n: n \geq 0\}$ is a $\pi$-base of
  	$X$. For any $n, k \in \mathbb N$, let
  	\[
  	S_{n,k} := \bigcup_{j\geq k} {\rm int}(f^{-j}(P_n)),
  	\]
    and 
    \[
    S := \bigcap_{n\in \mathbb N}\bigcap_{k\in \mathbb N}S_{n,k}.
  	\]
  	By (2), $S_{n,k}$ is open and dense in $X$ and thus $S$ is a dense $G_\delta$-set in $X$. Now 
  	we show that $\omega f(x)= X$ for all $x \in S$. To this end, let $x \in S$ and $k \in \mathbb N$
  	be fixed. For each nonempty open subset $V\subseteq X$, as $\{P_n: n \geq 0\}$ is a $\pi$-base, 
  	we can choose an $n_0\ge 0$ such that $P_{n_0}\subseteq V$. Then
  	\[
  	x \in S_{n_0, k} = \bigcup_{j\geq k} {\rm int}(f^{-j}(P_{n_0})) \subseteq 
  	\bigcup_{j \geq k}f^{-j}(P_{n_0}).
  	\]
  	Thus there is $n \ge k$ such that $x \in f^{-n}(P_{n_0})$, which implies that $f^n(x) \in V$. 
  	This means that ${\rm Orb}_f (f^k(x))$ is dense. It follows that $\omega f(x) =X$.
    \end{proof}

    \begin{theorem} \label{thm:equivalence}
    Let $(X, f)$ be a quasi-continuous dynamical system. Suppose that $X$ is a space of the second 
    category with a countable $\pi$-base ${\mathscr P}=\{P_n: n\in \mathbb N\}$. Then {\rm TT}$_+$ 
    implies that {\rm DO$_+$}. In addition, if $X$ is a Baire space, then {\rm Trans$_f$} 
    contains a dense $G_\delta$-set of $X$.
    \end{theorem}
  
    \begin{proof}
  	Let $(X, f)$ satisfy {\rm TT}$_+$. For each $n \in \mathbb N$, we define an open subset $U_n$
  	by
  	\[
  	U_n: = \bigcup_{k\in \mathbb N} {\rm int} \left(f^{-k}(P_n) \right).
  	\]
  	Let $U:= \bigcap_{n\in \mathbb N} U_n$. We first show that each $U_n$ is dense in $X$. To 
  	this end, let $V$ be an arbitrary nonempty open subset of $X$. Since $(X, f)$ satisfies
  	{\rm TT}$_+$, we have $N_+(V, P_n) \ne \emptyset$ and hence there is a 
  	$k \in \mathbb N$ such that $V \cap f^{-k}(P_n) \ne \emptyset$. 
  	Since $f^k$ is quasi-continuous, by Lemma \ref{lemma:key}  there is a nonempty open subset 
  	$W$ in $X$ such that $W \subseteq V \cap f^{-k}(P_n) $. This implies that $V \cap U_n \ne 
  	\emptyset$ which proves the claim.
  	
  	%\medskip
  	Next we show that $U \subseteq {\rm Trans}_f$. Let $x\in U$. We need to 
  	verify that Orb$_f(x)$ is dense in $X$. Let $G$ be an arbitrary nonempty open subset of $X$. 
  	Since $\mathscr P$ is a $\pi$-base, there must be some $n_0\in \mathbb N$ such that $P_{n_0} 
  	\subseteq G$. Then $x\in U_{n_0}$ implies that $f^{k_0}(x) \in P_{n_0} \subseteq G$ for some 
  	$k_0 \in \mathbb N$. Hence $x\in {\rm Trans}_f$.
  	
  	%\medskip
  	Finally, since $X$ is a space of the second category, $U \ne \emptyset$. It follows that
  	${\rm Trans}_f \ne \emptyset$ and thus $(X,f)$ satisfies DO$_+$. In addition, if $X$ is a
  	Baire space, then $U$ is a dense $G_\delta$-set of $X$.
    \end{proof}

    \begin{proposition} \label{prop:dense-orbit}
    Let $(X, f)$ be a quasi-continuous dynamical system whose phase space $X$ is a perfect and 
    $T_1$-space. Then, DO$_+$ and DO$_{++}$ are equivalent.
    \end{proposition}

    \begin{proof}
    Assume that $(X, f)$ satisfies DO$_+$. Then there is a point $x \in X$ such that ${\rm Orb}_f(x)$ 
    is dense in $X$. To show that $\omega f(x) = X$, let $z \in X$ and $k \in \N$, and let $U$ be 
    an open neighborhood of $z$. Since X is a perfect $T_1$-space, $U\setminus 
    \{x, f(x),\dots, f^k(x)\}$ is a nonempty open subset of $X$, and hence we can take a point 
    \[
    y \in \left(U \setminus \{x, f(x),\dots, f^k(x)\}\right) \cap {\rm Orb}_f(x).
    \] 
    Since $y \in {\rm Orb}_f (x)$, $y = f^m(x)$ for some $m \in \N$. Further, since 
    \[
    f^m(x) = y \not\in \{x, f(x),\dots, f^k(x)\},
    \] 
    we have $m > k$. Thus $y \in U \cap \{f^n(x) : n\ge k\}$, and hence $z\in \omega f(x)$. 
    Therefore, $(X,f)$ satisfies DO$_{++}$.
    \end{proof}
    
    As a corollary of Theorems \ref{thm:transitivity}-\ref{thm:equivalence} and Proposition
    \ref{prop:dense-orbit}, we obtain the following result which extends Theorem 2.1 in 
    \cite{crannell-martelli:00}. 
    
    \begin{corollary} \label{coro:equivalence}
    Let $(X, f)$ be a quasi-continuous dynamical system whose phase space $X$ is perfect and Hausdorff. 
    Assume in addition that $X$ is  also a Baire space with a countable $\pi$-base ${\mathscr P}=
    \{P_n: n\in \mathbb N\}$. Then {\rm TT, TT$_+$, TT$_{++}$, DO, DO$_+$} and {\rm DO$_{++}$} are 
    equivalent. Furthermore, each of these properties is equivalent to each of the following:
    \begin{enumerate}
    \item {\rm Trans$_f$} contains a dense $G_\delta$-set of $X$.
    \item The set $\{ x\in X: \omega f(x)=X\}$ contains a dense $G_\delta$-set of $X$.
    \end{enumerate}
    \end{corollary}
    
    To conclude this section, we summarize the relationships between the seven properties IN, TT, TT$_{+}$, 
    TT$_{++}$, DO, DO$_{+}$ and DO$_{++}$ in a quasi-continuous dynamical system by combining Diagram 1, 
    Theorems \ref{thm:transitivity} and \ref{thm:equivalence} into the following figure.
    
    \[
    \begin{tikzpicture}
    %\begin{scope}[>=latex]
    %\draw[->] (-2,0) -- (2,0); % x-axis
    %\draw[->] (0,-2) -- (0,2); % y-axis
    %\end{scope}
    
    \node (a) at (-2,2) {DO$_{++}$};
    \node (b) at (0,2)  {DO$_+$};
    \node (c) at (2,2)  {DO};
    \graph {(a) -> (b) -> (c)};
    \graph{(b) -> [bend right] (a)};
    
    \node at (-1, 2.8) {perfect $T_1$};
    
    \node (d) at (-2, 0.2) {TT$_{++}$};
    \node (e) at (0, 0.2) {TT$_{+}$};
    \node (f) at (2, 0.2) {TT};
    \node (g) at (4, 0.2) {IN};
    \graph {(d) -> (e) -> (f) <-> (g)};
    \graph{(f) -> [bend left] (e)};
    \graph{(e) -> [bend left] (d)};
    
    \node at (-0.8, -0.6) {perfect $T_2$};
    \node at (1.2, -0.6) {perfect $T_2$};
    
    \draw[->] (-2,1.5) -- (-2,0.7);
    \draw[->] (2,1.5) -- (2,0.7);
    \draw[->] (-0.2,0.7) -- (-0.2,1.5);
    
    \node at (0.5, 1.4) {c.$\pi$.b.};
    \node at (0.7, 1.0) {2nd cat.};
    
    \node (h) at (0.5, -1.5) {Diagram 2.};
    
    %\draw (-2,0) -- (2,0); % x-axis
    %\draw (0,-2) -- (0,2); % y-axis
    %\draw[dashed] (-0.2,1.7) -- (1.7,-0.2);
    %\draw[dashed] (-1,1) -- (1,-1);
    %\draw[dashed, ultra thick, red] (-1.6,0.2) -- (0.2,-1.6);
    %\fill (-0.71, -0.71) circle (2pt);
    %\coordinate [label=left:${\bf x}_*$] (x_*) at (-0.71,-0.8);
    %\draw (0,0) circle (1cm); %circle
    \end{tikzpicture}
    \]
    
    %%%%%%%%%%%%%%%%%%%%%%%%%%%%%%%%%%%%%%%%%%%%%%%%%% 
    \section{The case of imperfect phase spaces}
    %%%%%%%%%%%%%%%%%%%%%%%%%%%%%%%%%%%%%%%%%%%%%%%%%%
    
    In this section, we discuss what happens in a quasi-continuous dynamical system $(X,f)$,
    when the phase space $X$ contains isolated points. Throughout this section, we 
    assume that $X$ is a Hausdorff space.
    
    Let ${\rm Iso}_X$ be the set of isolated points in $X$. Similar to the analysis in Section 5 of 
    \cite{akin-carlson:2012}, we discuss how ${\rm Iso}_X$ sits in $X$ by analyzing the preimages of 
    isolated points and then see what role ${\rm Iso}_X$ plays in the study of topological transitivity. 
    Through this analysis, we derive a conclusion similar to that of \cite{akin-carlson:2012} 
    but our results are an extension of those in Section 5 of \cite{akin-carlson:2012}.
    
    \begin{lemma} \label{lem:null_preimage}
    If $(X, f)$ satisfies {\rm TT}, then there is at most one point $x\in {\rm Iso}_X$ such that 
    $f^{-1}(x) =\emptyset$.
    \end{lemma}

    \begin{proof}
    If $f(X)$ is dense in $X$, then $f^{-1}(x)\ne \emptyset$ for any $x \in {\rm Iso}_X$. Assume that
    $f(X)$ is not dense in $X$. Then $X \setminus \overline{f(X)}$ is nonempty open. By an argument 
    similar to that of Corollary \ref{coro:equivalence}, we can show that $X \setminus \overline{f(X)}$ 
    contains only one point $x$. Then $x$ must be an isolated point such that $f^{-1}(x) =\emptyset$.
    \end{proof}
    
    %Note that if $x, y \in {\rm Iso}_X$, then either $x\in {\rm orb}_f(y)$ or $y\in {\rm orb}_f(x)$. 
    %Thus, ${\rm Iso}_X \subseteq {\rm orb}^{\pm}_f(x)$ for any $x \in {\rm Iso}_X$.
    
    The following characterization of a quasi-continuous mapping between two topological spaces is 
    useful.
    
    \begin{lemma}[\cite{Neubrunn:1988}, \cite{Neubrunnova:1973}] \label{lem:charac}
    Given topological spaces $X$ and $Y$, a mapping $f: X \to Y$ is quasi-continuous 
    on $X$ if, and only if, 
    \[
    f^{-1}(V) \subseteq \overline{{\rm int} \left(f^{-1}(V)\right)}
    \]
    for any open subset $V$ of $Y$.
    \end{lemma}
    
    \begin{lemma} \label{lem:periodic_iso}
    Let $(X,f)$ be a quasi-continuous dynamical system satisfying {\rm TT}.
    If $x \in {\rm Iso}_X$ and $|f^{-1}(x)| >1$, then $x$ is periodic and $|f^{-1}(x)|=2$. Thus
    $f^{-1}(x)$ is a finite open set consisting entirely of isolated points.
    \end{lemma}

    \begin{proof}
    If $x \in {\rm Iso}_X$, by Lemma \ref{lem:charac} $f^{-1}(x)\subseteq\overline{{\rm int}f^{-1}(x)}$. 
    In addition, if $|f^{-1}(x)| >1$ also holds, there are two distinct points $y, z \in f^{-1}(x)$.
    Since $X$ is Hausdorff, we can choose two disjoint open subsets $U'$ and $V'$ in $X$ such that
    $y \in U'$ and $z\in V'$. Put $U := U'\cap {\rm int}f^{-1}(x)$ and $V := V'\cap {\rm int}f^{-1}(x)$.
    Then, $U$ and $V$ are nonempty open subsets in $X$ such that $U \cap V =\emptyset$ and $U \cup V 
    \subseteq f^{-1}(x)$. Further, as $(X,f)$ satisfies TT, we can require $N_+(U, V) \ne \emptyset$ 
    without loss of generality.
    
    Let $k$ be the smallest element of $N_+(U,V)$. Since $U\cap V =\emptyset$, $k>0$. Choose $y \in U$
    such that $f^k(y) \in V$. Then, $x=f(y)$ and $f^k(y) \in f^{-1}(x)$. It follows that
    \[
    x = f(f^k(y)) = f^k(f(y)) =f^k(x).
    \]
    Thus $x$ is periodic. Indeed, by the minimality of $k$, we conclude that $k$ is the period of $x$ 
    and the forward orbit of $x$ is 
    \[
    {\rm Orb}_f(x) =\{x, f(x), f^2(x), \cdots, f^{k-1}(x) \}.
    \]
    
    Suppose $|f^{-1}(x)| >2$. Then, by shrinking $U$ and $V$ if necessary, we can choose a nonempty 
    open subset $W$ of $X$, which is disjoint from both $U$ and $V$, such that 
    \[
    U\cup V \cup W \subseteq f^{-1}(x). 
    \]
    Since $(X,f)$ satisfies TT, there exists an integer $m\in \N$ such that 
    \[
    m \in N_+(U, W) \cup N_+(W, U).
    \] 
    Without loss of generality, let $m \in N_+(U, W)$. The disjointness of $U$ and $W$ implies $m>0$. 
    We choose a point $z \in U$ such that $f^m(z) \in W$. Similar to the previous argument, we have
    \[
    x =f(f^m(z)) = f^m(f(z)) = f^m(x).
    \]
    Hence $m$ must be a multiple of $k$. This implies that
    \[
    f^m(z) = f^{m-1}(x) =f^{k-1}(x) \in V,
    \]
    which contradicts the fact that $f^m(z) \in W$.
    
    Finally, since the cardinality of $f^{-1}(x)$ is 0, 1 or 2, then $f^{-1}(x)$ is a finite set. 
    If it has cardinality 1 or 2,  then ${\rm int}\left(f^{-1}(x)\right)$ nonempty in addition to $X$  
    Hausdorff (in the case of cardinality 2) implies that every point of $f^{-1}(x)$ is isolated. 
    \end{proof}

    \begin{lemma} \label{lem:preimage_two}
    Let $(X,f)$ be a quasi-continuous dynamical system satisfying {\rm TT}. Then there is at most one 
    $x \in {\rm Iso}_X$ such that $|f^{-1}(x)|=2$.
    \end{lemma}
    
    \begin{proof}
    Suppose that there are two distinct points $x, x' \in {\rm Iso}_X$ such that 
    \[
    |f^{-1}(x)|=|f^{-1}(x')|=2. 
    \]
    By Lemma \ref{lem:periodic_iso}, both $x$ and $x'$ are periodic, and ${\rm Orb}_f(x) = {\rm Orb}_f(x')$
    since $(X, f)$ satisfies TT. Note that there are points $y$ and $y'$ such that 
    $y\in f^{-1}(x) \setminus {\rm Orb}_f(x)$ and $y'\in f^{-1}(x') \setminus {\rm Orb}_f(x')$.
    By Lemma \ref{lem:periodic_iso}, both $y$ and $y'$ are isolated. Since $(X,f)$ satisfies TT,
    $N(\{y\}, \{y'\}) \ne \emptyset$ and there is $m \in \N$ such that $f^m(y)=y'$ without loss of
    generality. Since $x\ne x'$ and $y \in f^{-1}(x)$, we have $m>0$ and 
    \[
    y' = f^m(y) = f^{m-1} (f(y)) = f^{m-1}(x),
    \]
    which contradicts 
    \[
    y' \not \in {\rm Orb}_f(x') = {\rm Orb}_f(x).
    \]
    Hence there is at most one isolated point $x$ in $X$ with $|f^{-1}(x)| =2$.
    \end{proof}

    Lemmas \ref{lem:null_preimage}, \ref{lem:periodic_iso} and \ref{lem:preimage_two} tell us that 
    in a quasi-continuous dynamical system $(X,f)$ with the property TT, the cardinality of the preimage 
    of any isolated point is 0, 1 or 2. Further, such a system has at most one isolated point with 
    empty preimage and at most one isolated point whose preimage contains exactly two points.
    
    \begin{proposition} \label{prop:trans_point}
    Let $(X,f)$ be a quasi-continuous dynamical system with isolated points. If $(X,f)$ satisfies {\rm TT}, 
    then ${\rm Trans}_f \subseteq {\rm Iso}_X$.
    \end{proposition}
    
    \begin{proof}
    Let $y \in {\rm Trans}_f$. Then ${\rm Orb}_f(y)$ is dense in $X$. Let $x$ be any isolated point. 
    Then there is an integer $k \in \N$ such that $x = f^k(y)$. If $k=0$, then $y=x \in {\rm Iso}_X$. 
    So we assume that $k>0$. This implies that $y \in f^{-k}(x)$. Applying Lemma 
    \ref{lem:periodic_iso} in finitely many steps, we conclude that $y \in {\rm Iso}_X$. Hence
    ${\rm Trans}_f \subseteq {\rm Iso}_X$.
    \end{proof}
    
    Proposition \ref{prop:trans_point} tells us that we should search for transitive points within those
    isolated points. We shall achieve this in the next result.
    
    \begin{theorem} \label{thm:imperfect-equivalence}
    Let $(X,f)$ be a quasi-continuous dynamical system with isolated points. Assume that $(X,f)$ 
    satisfies {\rm TT}.
    
    \medskip
    \begin{enumerate}
    \item If there is a (unique) point $x \in {\rm Iso}_X$ such that $f^{-1}(x) =\emptyset$, then $(X,f)$
    satisfies {\rm DO$_+$} and $x$ is the unique transitive point. In this case, none of {\rm TT$_+$, 
    TT$_{++}$} or {\rm DO$_{++}$} hold.
    
    \item If $|f^{-1}(x)| =1$ for every $x \in {\rm Iso}_X$, then either all of {\rm DO$_+$, TT$_+$, 
    TT$_{++}$} and {\rm DO$_{++}$} hold, or none of them hold. 
    
    \item If $f^{-1}(z) \ne \emptyset$ for every point $z \in {\rm Iso}_X$ and there is a (unique) 
    point $x \in {\rm Iso}_X$ such that $|f^{-1}(x)|=2$, then none of {\rm DO$_+$, TT$_+$, 
    DO$_{++}$} and {\rm TT$_{++}$} hold.
    \end{enumerate}
    \end{theorem}

    \begin{proof}
    (1) Since $(X,f)$ satisfies TT and $f^{-1}(x) =\emptyset$, for any nonempty open subset $U$ of $X$
    we must have $N_+(\{x\}, U) \ne \emptyset$. This implies that ${\rm Orb}_f(x)$ is dense in $X$ and 
    ${\rm Iso}_X \subseteq {\rm Orb}_f(x)$. Thus $(X,f)$ satisfies DO$_+$ and $x$ is a transitive 
    point.
    
    Since $f^{-1}(x) =\emptyset$, we cannot have $x\in {\rm Orb}_f(y)$ for any $y\in {\rm Iso}_X$ and 
    $y \ne x$. This means that any isolated point distinct from $x$ (if such a point exists) cannot 
    be transitive. Hence, $x$ is the unique transitive point.
    
    Pick any point $y\in X$ but $y \ne x$ ($y$ is not necessarily to be an isolated point). Since $X$ 
    is a Hausdorff space, there is an open set $U$ such that $y \in U$ and $x\not \in U$. As $f^{-1}(x) 
    =\emptyset$, we have $N_+(U, \{x\}) =\emptyset$. Thus, TT$_+$ or any property stronger than TT$_+$ 
    does not hold.
    
    %\medskip
    (2) We consider two subcases.
    
    \emph{Subcase 1.} There is a periodic isolated point $x\in {\rm Iso}_X$. Then
    \[
    {\rm Orb}_f(x) =\{ x, f(x), \cdots, f^{k-1}(x)\}
    \]
    for some $k >0$. By Lemma \ref{lem:periodic_iso}, we have ${\rm Orb}_f(x) \subseteq {\rm Iso}_X$. 
    This and the fact that $|f^{-1}(z)|=1$ for every $z \in {\rm Iso}_X$ imply that
    ${\rm Orb}_f(x) =\{ f^{n}(x): n \in \Z\}$.
    Now let $U$ be any nonempty open subset of $X$. Since $(X,f)$ satisfies TT, we have some
    integer $k \in N\left(\{x\}, U \right)$. Then
    \[
    f^{k}(x) \in U \cap \{ f^n(x): n \in \Z\} = U \cap {\rm Orb}_f(x).
    \]
    This implies that ${\rm Orb}_f(x)$ is dense in $X$. Since ${\rm Orb}_f(x)$ is closed in $X$,
    we have $X = {\rm Orb}_f(x)$. Thus, in this subcase, all of DO$_+$, TT$_+$, 
    TT$_{++}$ and DO$_{++}$ hold.
    
    \emph{Subcase 2.} ${\rm Iso}_X$ does not contain any periodic point. Let $x
    \in {\rm Iso}_X$. By Lemma \ref{lem:charac}, $f^{-1}(x)$ is open in $X$. Since $x$ is not
    a periodic point, $N_+(\{x\},  f^{-1}(x)) = \emptyset$. Then TT$_+$ and thus TT$_{++}$ do
    not hold. 
    
    To show that ${\rm Trans}_f =\emptyset$, by Proposition \ref{prop:trans_point},
    it suffices to show that $y \not\in {\rm Trans}_f$ for every $y \in {\rm Iso}_X$. Let $y \in 
    {\rm Iso}_X$. Since $|f^{-1}(y)|=1$, there is a point $y'$ such that $f(y') = y$, and 
    $y' \in{\rm Iso}_X$ by Lemma \ref{lem:periodic_iso}. Since $y$ is not a periodic point, $y' 
    \not \in {\rm Orb}_f(y)$. Thus $y' \not \in \overline{{\rm Orb}_f(y)}$ and hence $y \not \in 
    {\rm Trans}_f$. Therefore ${\rm Trans}_f=\emptyset$. This means that DO$_+$ and thus DO$_{++}$ 
    do not hold.
    
    %\medskip
    (3) By Lemma \ref{lem:periodic_iso}, $x$ is a periodic point with period $k >0$. Let
    \[
    {\rm Orb}_f(x) =\{ x, f(x), \cdots, f^{k-1}(x)\}
    \]
    and $f^{-1}(x) = \{ f^{k-1}(x), y\}$ with $y \ne f^{k-1}(x)$. Then $y \in {\rm Iso}_X$ and 
    ${\rm Orb}_f(x) \subseteq {\rm Iso}_X$. If $y = f^m(x)$ for some nonnegative integer $m <k-1$,
    then $x = f^{m+1}(x)$, which contradicts the fact that $k$ is the period of $x$. Thus,
    $y \notin {\rm Orb}_f(x)$.  (As a notational convenience, we identify the singleton $f^{-n}(y)$ with its single element to allow for ease of expression in what follows.) Indeed, by Lemma \ref{lem:periodic_iso}, we have that $f^{-n}(y) 
    \in {\rm Iso}_X \setminus {\rm Orb}_f(x)$ for all $n \in \N$. This implies that 
    $N_+(\{x\}, \{y\}) =\emptyset$. Therefore TT$_+$ and hence TT$_{++}$ do not hold. 
    
    Next we show that DO$_+$ does not hold, that is, ${\rm Trans}_f =\emptyset$. For any point 
    $z\in {\rm Iso}_X$, by the TT property, we have $N(\{z\}, \{y\}) \ne \emptyset$. If $N_+(\{z\}, 
    \{y\}) \ne \emptyset$, then $y=f^m(z)$ for some $m \in \N$. In this case, $z \in \{ f^{-n}(y): 
    n \in \N\}$. If $N_+(\{y\}, \{z\}) \ne \emptyset$, then $z=f^m(y)$ for some $m \in \N$. So, 
    we have either $z=y$ or $z=f^n(x)$ for some $n \in \N$. This means that either $z \in \{ 
    f^{-n}(y): n \in \N\}$ or $z \in  {\rm Orb}_f(x)$ holds. We have just verified that 
    \[
    {\rm Iso}_X = {\rm Orb}_f(x) \cup \{f^{-n} (y): n \in \N\}.
    \]
    Note that no point in ${\rm Orb}_f(x)$ is transitive, as $y \notin {\rm Orb}_f(x)$. Now we
    consider a point $z \in {\rm Iso}_X \setminus {\rm Orb}_f(x)$. Then $z= f^{-n}(y)$ for some 
    $n \in \N$. By Lemma \ref{lem:periodic_iso} and assumption, $f^{-1}(z)$ is an isolated point. 
    Since $f^{-m}(y)\ne f^{-n}(y)$ for all $m\ne n$ and $m, n\in \N$, $f^{-1}(z) \notin 
    {\rm Orb}_f(z)$. It follows that $z \not \in {\rm Trans}_f$. Hence, ${\rm Trans}_f =\emptyset$ 
    by Proposition \ref{prop:trans_point}. Consequently, DO$_+$ and hence DO$_{++}$ do not hold. 
    \end{proof}

    \begin{corollary} \label{coro5.1}
    Let $(X,f)$ be a quasi-continuous dynamical system with isolated points. Then, {\rm DO$_+$, 
    TT$_+$, DO$_{++}$} and {\rm TT$_{++}$} are equivalent.
    \end{corollary}

    \begin{corollary} \label{coro5.2}
    Let $(X,f)$ be a quasi-continuous dynamical system with isolated points. Then,
    {\rm TT} and {\rm DO} are equivalent.
    \end{corollary}
    
    \begin{proof}
    First we know that  DO implies TT in any dynamical system. Assume then that $(X,f)$ satisfies the
    TT property. We consider the three cases listed in Theorem \ref{thm:imperfect-equivalence}.
    
    In case (1), as shown in the proof of Theorem \ref{thm:imperfect-equivalence}, $(X,f)$ 
    satisfies DO$_+$ and thus DO. 
    
    In case (2), $|f^{-1}(x)|=1$ for all $x \in {\rm Iso}_X$. Then for any $x \in {\rm Iso}_X$, 
    $f^{-n}(x)$ is an isolated point for all $n\in \N$. The TT property implies that
    $\left\langle f^n(x): n \in \Z\right \rangle$ is an orbit sequence dense in $X$. Thus $(X,f)$ 
    satisfies DO.
    
    In case (3), let $x \in {\rm Iso}_X$ be the unique point such that $|f^{-1}(x)|=2$. Note that
    $x$ is a periodic point with period $k>0$. Let $f^{-1}(x) =\{ f^{k-1}(x), y\}$. Then
    $f^{-n}(y)\in {\rm Iso}_X$ for all $n \in \N$. Again, the TT property implies that $\langle 
    f^n(y): n \in \Z \rangle$ is an orbit sequence dense in $X$. Hence $(X,f)$ satisfies DO.
    \end{proof}

    \vskip 1em
    \noindent
    \textbf{Acknowledgement.} The authors would like to thank the referee for her/his 
    thorough check of the original manuscript. Her/his valuable comments and suggestions have improved 
    the presentation of this paper. In particular, Example \ref{exam:quas-cont} was provided by her/him.
   
    %\vskip 2em
    %%%%%%%%%%%%%%%%%%%%%%%%

    \end{document}